\documentclass[a4paper,11pt,reqno]{amsart}

\usepackage{amsmath}
\usepackage{enumerate}
\usepackage{amsfonts}
\usepackage{gensymb}
\usepackage{wasysym}
\usepackage{amsthm}
\usepackage{enumerate}
\usepackage{amssymb}
\usepackage{tikz}
\usetikzlibrary{trees}
\usetikzlibrary{arrows}
\usepackage{stmaryrd}
\usepackage{slashbox}
\usepackage{fullpage}
\usepackage{nicematrix}
\usetikzlibrary{automata, positioning, arrows}

\newcommand{\E}{{\mathcal E}}
\newcommand{\G}{{\mathcal G}}

\newcommand{\N}{\mathbb N}

\renewcommand{\dim}{\mathrm{dim}}
\newcommand{\ddim}{\mathrm{Dim}}
\newcommand{\sdim}{\mathrm{sdim}}
\newcommand{\halts}{{\downarrow}}

\newcommand{\llb}{\llbracket}
\newcommand{\rrb}{\rrbracket}
\newcommand{\uh}{{\upharpoonright}}

\newtheorem{thm}{Theorem}[section]

\theoremstyle{definition}
\newtheorem{dfn}[thm]{Definition}

\theoremstyle{remark}

\usepackage{xcolor}	
	\usepackage{soul}

\title{Length Functions and the Dimension of Points in Self-Similar Fractal Trees}
\author{Christopher P.\ Porter}

\begin{document}

\maketitle

\begin{abstract}  
In this paper, we study the effective dimension of points in infinite fractal trees generated recursively by a finite tree over some alphabet.  Using unequal costs coding, we associate a length function with each such fractal tree and show that the channel capacity of the length function is equal to the similarity dimension of the fractal tree (up to a multiplicative constant determined by the size of the alphabet over which our tree is defined).  Using this result, we derive formulas for calculating the effective dimension and strong effective dimension of points in fractal trees, establishing analogues of several results due to Lutz and Mayordomo, who studied the effective dimension of points in self-similar fractals in Euclidean space.  Lastly, we explore the connections between the channel capacity of a length function derived from a finite tree and the measure of maximum entropy on a related directed multigraph that encodes the structure of our tree, drawing on work by Abram and Lagarias on path sets, where a path set is a generalization of the notion of a sofic shift.

\end{abstract}

\section{Opening}

The tools from algorithmic information theory, particularly the notion of effective dimension of an individual point, have found successful application in the study of fractal geometry.  As a particularly striking example, Lutz and Mayordomo in \cite{LutMay08} provide a general formula for calculating the effective dimension $\dim(x)$ of a point $x$ in a self-similar fractal $F$ in $\mathbb{R}^n$:
\[
\dim(x)=\sdim(F)\dim^\mu(y),
\]
where $\sdim(F)$ is the similarity dimension of $F$, $\dim^\mu(\cdot)$ is an effective analogue of the Billingsley dimension with respect to a specific probability measure $\mu$ defined in terms of the fractal $F$, and $y$ is an infinite sequence over some finite alphabet that serves as a code for $x$ as an element of $F$ (we will define all of these notions in Section \ref{sec-coding} below).  They further established an analogous for effective strong dimension $\ddim(x)$ of points $x\in F$.

In this study, we prove similar results for points in infinite self-similar trees over some finite alphabet.  Our main proof relies on the machinery of coding with unequal costs from information theory, which in our context, amounts to considering notions of algorithmic information theory in terms of length functions that do not necessarily measure the length of a string of symbols merely in terms of the number of symbols in the string.  More generally, we will identify, for a given $m$-ary tree $T$ for some $m\in\N$, the relationship between the following quantities from seemingly disparate areas:

\begin{enumerate}
\item $\alpha$, the channel capacity of a code with unequal letter costs that is defined in terms of $T$,
\item $\sdim(F)$, the similarity dimension of an infinite self-similar tree $F$ generated by $T$, and
\item $\log(\lambda)$, the negative logarithm of the Perron eigenvalue of the adjacency matrix of a specific directed graph $G$ determined by $T$,
\end{enumerate}
are given by the equalities
\[
\alpha=\log(m)\cdot\sdim(F)=\log(\lambda)
\]
(where the logarithm here and henceforth is taken to the base 2).
In the case that $T$ is a binary tree, we get the equality of the three quantities (1)-(3) listed above.
From the relationship between (1) and (2), we will derive our analogue of the above-mentioned Lutz/Mayordomo result that holds for points in self-similar trees.  Using (3), we show how the measure $\mu$ used in the calculation of the effective version of Billingsley dimension mentioned above can be obtained from a transformation of the unique measure of maximal entropy on a specific subshift related to the tree $T$ (in this case, the Parry measure on a specific sofic shift).

The outline of the remainder of this article is as follows.  In Section \ref{sec-background}, we provide background material, including the basics of effective dimension.  In Section \ref{bigger-picture}, we lay out a brief survey of previous work at the intersection of algorithmic information theory, fractal geometry, and symbolic dynamics in order to provide a broader context for the reader. Section \ref{sec-coding} begins the original contributions of this work, in which we establish the above-described relationship between the channel capacity of a length function generated from a finite tree $T$ and the similarity dimension of the self-similar fractal tree generated from $T$.  Using this relationship, we derive the analogue of the Lutz/Mayordomo formulas for the dimension and strong dimension of points in the self-similar fractal generated from $T$.   Finally, in Section \ref{sec-dynamics}, we take a dynamical systems perspective on self-similar fractal trees.  In particular, we show how to obtain a directed multigraph $G$ from a given finite tree $T$, where the one-sided infinite walks through $G$ are precisely the infinite paths through the infinite self-similar fractal $F$ generated by $T$, thereby establishing the equivalence between channel capacity of the length function with the value $\log(\lambda)$ described in (3) above.  Using properties of $G$, we derive the measure of maximum entropy on the closure of $F$ under the shift operator and prove that this measure is equivalent to one naturally defined in terms of the length function associated with the tree $T$.

\section{Background} \label{sec-background}

For $k\in\N$, $\Sigma_k^*$ consists of all finite strings over the alphabet $\Sigma_k=\{0,1,\dotsc,k-1\}$.  Similarly, $\Sigma_k^\infty$ consists of all infinite sequences over the same alphabet. We will write elements of $\Sigma_k^*$ as $\sigma,\tau,$ and so on, while elements of $\Sigma_k^\infty$ will be written as $x,y,z$, and so on.  The concatenation of strings $\sigma,\tau\in\Sigma_k^*$ is written as $\sigma^\frown\tau$.  The standard length of a string $\sigma\in\Sigma_k^*$ is written as $|\sigma|$ (we will consider more general length functions in Section \ref{sec-coding}).  Given $\sigma,\tau\in \Sigma_k^*$, we write $\sigma\preceq\tau$ to indicate that $\sigma$ is an initial segment of $\tau$; similarly, if $\sigma$ is an initial segment of some $x\in \Sigma_k^\infty$, we write $\sigma\prec x$.   Given $x\in \Sigma_k^\infty$ and $i\in\mathbb{N}$, $x\uh i$ is the initial segment of $x$ of length $i$.  More generally, given $x\in \Sigma_k^\infty$ and $m,i\in\mathbb{N}$, $x\uh[m,m+i)$ is the string $\tau\in\Sigma_k^*$ of length $i$ such that $\tau(j)=x(m+j)$ for $j=0,\dotsc,i-1$.   A set $S\subseteq \Sigma_k^*$ is \emph{prefix-free} if for $\sigma,\tau\in S$, $\sigma\preceq\tau$ implies that $\sigma=\tau$.

The trees that we consider here are subsets of $\Sigma_k^*$ that are closed downwards under $\prec$. Given a prefix-free set $S\subseteq\Sigma_k^*$, we can define a tree $T$ by closing $S$ downwards under $\prec$.  In this case, in an abuse of terminology, we will refer to the members of $S$ as the \emph{terminal nodes} of $T$; any element of $T$ that is not a terminal node will be referred to as a \emph{non-terminal node} of $T$.  Hereafter, we will specify a finite tree simply by specifying its set of terminal nodes.  Given an infinite tree $T\subseteq\Sigma_k^*$, we will write $[T]$ as the set of infinite paths through $T$, i.e., $[T]=\{x\in\Sigma_k^\infty\colon (\forall n)\;x\uh n\in T\}$.

For $\sigma\in\Sigma_k^*$, the \emph{cylinder set defined by $\sigma$} is the set $\llb\sigma\rrb=\{x\in \Sigma_k^\infty\colon \sigma\prec x\}$.  Let $(\sigma_i)_{i\in\N}$ be the enumeration of $\Sigma_k^*$ in length-lexicographical order. A set $S\subseteq \Sigma_k^\infty$ is \emph{effectively open} (or a \emph{$\Sigma^0_1$ class}) if there is some computable function $f:\N\rightarrow\N$ such that $S=\bigcup_{i\in\N}\llb\sigma_{f(i)}\rrb$.  $P\subseteq \Sigma_k^\infty$ is \emph{effectively closed} (or a \emph{$\Pi^0_1$ class}) if $\Sigma_k^\infty\setminus P$ is effectively open.  Equivalently, $P$ is a $\Pi^0_1$ class if $P=[T]$ for some infinite computable tree $T\subseteq \Sigma_k^*$.

A measure $\mu$ on $\Sigma_k^\infty$ is determined by its values on cylinder sets.  Hereafter, for $\sigma\in \Sigma_k^*$, we will write $\mu(\llb\sigma\rrb)$ as $\mu(\sigma)$.
$\mu$ is a \emph{computable} measure if $\mu(\sigma)$ can be computably approximated to an arbitrary precision uniformly in $\sigma\in\Sigma_k^*$.  A measure $\mu$ on $\Sigma_k^\infty$ is a \emph{Bernoulli measure} if $\mu(\sigma^\frown i)=\mu(\sigma)\mu(i)$ for all $\sigma\in \Sigma_k^*$ and $i\in\Sigma_k$.

We assume the reader is familiar with the basics of computability theory and algorithmic randomness; see, for instance, \cite{Nie09}, \cite{DowHir10}, or \cite{SheUspVer17}.  See also \cite{FraPor20} for an recent survey on algorithmic randomness.

For $m,k\in\N$, a Turing functional $\Phi:\Sigma_k^\infty\rightarrow\Sigma_m^\infty$ is an effective map defined in terms of a computable function $\phi:\Sigma_k^*\rightarrow\Sigma_m^*$ satisfying the property that for $\sigma,\tau\in \Sigma_k^*$, $\sigma\preceq\tau$ implies $\phi(\sigma)\preceq\phi(\tau)$.  For such a function $\phi$, we can define $\Phi$ on $x\in\Sigma_k^\infty$ by setting $\Phi(x)=\lim_{n\rightarrow\infty}\phi(x\uh n)$ (where this limit is the longest element of $\Sigma_m^*\cup\Sigma_m^\infty$ that has $\phi(x\uh n)$ as an initial segment for all $n\in\mathbb{N}$).  In our context, we will only consider \emph{total} Turing functionals, i.e., functionals given in terms of a computable function $\phi$ for which $\lim_{n\rightarrow\infty}|\phi(x\uh n)|=\infty$ for all $x\in\Sigma_k^\infty$.

Notions of the effective dimension of a sequence $x\in\Sigma_k^\infty$ were introduced by Lutz in \cite{Lut00} and further developed in, for instance, \cite{May02},  \cite{Lut03}, and \cite{AthHitLut04}.  Lutz originally defined effective dimension in terms of certain betting strategies he referred to as gales, but a characterization of dimension in terms of Kolmogorov complexity was later given by Mayordomo (in \cite{May02}).  We will use this latter characterization in the present study.


Fix $m\in\N$. Let $M:\Sigma_2^*\rightarrow \Sigma_m^*$ be a prefix-free Turing machine (recall that a Turing machine $M$ is prefix-free if the domain of $M$ forms a prefix-free set).
For $\tau\in \Sigma_m^*$, the \emph{prefix-free Kolmogorov complexity of $\tau$ with respect to $M$} is defined to be $K_M(\tau)=\min\{|\sigma|: M(\sigma)=\tau\}$.  Let $(M_i)_{i\in\N}$ be an effective enumeration of all prefix-free machines.  We can define a universal prefix-free machine $U:\Sigma_2^*\rightarrow \Sigma_m^*$ by setting, for each $e\in\N$ and $\sigma\in\Sigma_n^*$, $U(1^e0\sigma)=M_e(\sigma)$ when the latter is defined (otherwise, $U(1^e0\sigma)$ is undefined).  For $\tau\in\Sigma_m^*$, we then define the \emph{prefix-free Kolmogorov complexity of $\tau$} to be $K(\tau):=K_U(\tau)$.  As shown by Kolmogorov \cite{Kol65}, for every prefix-free machine $M$, $K(\sigma)\leq K_M(\sigma)+O(1)$ for all $\sigma\in\Sigma_m^*$, where the additive constant only depends on $M$. Using Kolmogorov complexity, we can define two notions of effective dimension as follows.  For $x\in \Sigma_m^\infty$,
\[
\dim(x)=\liminf_{n\rightarrow\infty}\frac{K(x\uh n)}{n\log(m)}
\]
and 
\[
\ddim(x)=\limsup_{n\rightarrow\infty}\frac{K(x\uh n)}{n\log(m)},
\]
where the former notion is known as the \emph{effective dimension} of $x$ and the latter notion is know as the \emph{effective strong dimension} of $x$.

In \cite{LutMay08}, Lutz and Mayordomo introduced two effective notions of Billingsley dimension that generalize the above two notions.  For a computable measure $\mu$ on $\Sigma_m^\infty$ and $x\in\Sigma_m^\infty$, we have
\[
\dim^\mu(x)=\liminf_{n\rightarrow\infty}\frac{K(x\uh n)}{-\log\mu(x\uh n)}
\]
and
\[
\ddim^\mu(x)=\limsup_{n\rightarrow\infty}\frac{K(x\uh n)}{-\log\mu(x\uh n)},
\]
the \emph{effective dimension of $x$ with respect to $\mu$} and the \emph{effective strong dimension of $x$ with respect to $\mu$}, respectively.

%

\section{The bigger picture}\label{bigger-picture}

Before we turn to the main contributions of this study, we first offer a brief survey of the key developments that lie at the intersection of algorithmic information theory, fractal geometry, and symbolic dynamics, which should in turn provide the reader with a broader context to understand and appreciate the results contained herein.

The relationship between Kolmogorov complexity and various notions of entropy was observed early in the development of algorithmic information theory.  Indeed, as early as 1970, merely five years after Kolmogorov first offered his definition of complexity in \cite{Kol65}, Zvonkin and Levin \cite{ZvoLev70} proved that if we consider binary sequences produced as the result of a sequence of independent, identically distributed random variables with probabilities $(p,1-p)$ for some computable $p\in(0,1)$, for $\mu$-almost every sequence $x\in\Sigma_2^\infty$, where $\mu$ is the Bernoulli measure on $\Sigma_2^\infty$ with parameters $(p,1-p)$, we have 
\[
\lim_{n\rightarrow\infty}\frac{K(x\uh n)}{n}=H(p)=-p\log(p)-(1-p)\log(1-p)\]
(in fact, Zvonkin and Levin showed the stronger result that the above expression holds for all sequences that are Martin-L\"of random with respect to the measure $\mu$; the details of this notion of randomness is not necessary for this study, but see, for instance, \cite[Chapter 3]{Nie09}, \cite[Chapter 6]{DowHir10}, or \cite[Chapter 3]{SheUspVer17}).  A similar result was obtained soon thereafter by Heim \cite{Hei79} using a variant of Kolmogorov complexity defined in terms of block codes.  Subsequent work by Brudno \cite{Bru82} and White \cite{Whi93} characterized the complexity of the codes of trajectories of points in a dynamical system on a compact space using a complexity measure defined explicitly in terms of the asymptotic behavior of $K(x\uh n)/n$ (which we do not define here), where $x$ is the code of a trajectory of a point in the space in question.  In this setting, for an ergodic measure $\mu$, the complexity for $\mu$-almost every point what shown to be equal to the Kolmogorov-Sinai entropy of $\mu$.

The first explicit results on algorithmic notions of fractal dimension are due to Lutz \cite{Lut00}, who defined an effective version of Hausdorff dimension applicable to both individual sequences and collections of sequences using certain betting strategies known as gales.  Soon after Lutz's initial contribution, Mayordomo \cite{May02} characterized Lutz's effective notion of Hausdorff dimension in terms of Kolmogorov complexity (a characterization we have used in our definition of effective dimension given at the end of the previous section).  In addition, Athreya, Hitchcock, Lutz, and Mayordomo \cite{AthHitLut04} defined an effective version of packing dimension and provided alternative characterizations of the notion in terms of both gales and Kolmogorov complexity (with the latter characterization is the definition of strong effective dimension given at the end of the previous section).  

With this initial framework of effective notions of fractal dimension in place, a considerable amount of work followed.  Highlights relevant to this study include:
\begin{itemize}
\item Simpson's work \cite{Sim15} on the effective Hausdorff dimension of subshifts $X\subseteq A^G$ (where $A$ is a finite set of symbols and $G=\mathbb{N}^d$ or $G=\mathbb{Z}^d$ for some $d\in\mathbb{N}$), which established that the effective dimension of $X$ is equal to the topological entropy of $X$;
\item the extension of effective notions of dimension to Euclidean space, for instance to study dimension of points in a self-similar fractal by Lutz and Mayorodomo \cite{LutMay08}, as discussed in the introduction, as well as the effective dimension of points in a random subfractal of a self-similar fractal \cite{GuLutMay14};
\item the improvement by Hoyrup \cite{Hoy12} of the original Zvonkin-Levin result described above to the context of computable shift-invariant, ergodic measures;
\item the study of the dimensions of points on various lines in the plane by N.\ Lutz and Stull \cite{LutStu17}, \cite{LutStu20};
\item the use of point-to-set principles for effective notions of dimension to characterize the classical dimension of various sets.
\end{itemize}
This latter point bears further unpacking.  As shown by Lutz and N.\ Lutz \cite{LutLut18}, for every $E\subseteq \mathbb{R}^n$, the classical Hausdorff dimension $\dim_H(E)$ can be calculated from an analogue of the effective dimension of the points in $E$ via the expression
\[
\dim_H(E)=\min_{A\subseteq \mathbb{N}}\sup_{x\in E}\dim^A(x),
\]
where $\dim^A(x)$ is the effective dimension of $x$ relative to the oracle $A$; a similar expression holds for classical packing dimension.  These point-to-set principles have proven to be powerful tools, allowing for the solution of open problems in classical fractal geometry, for instance, problems involving the intersections and products of fractals in Euclidean space due to N.\ Lutz \cite{Lut21} and a generalization of Marstand's projection theorem due to N.\ Lutz and Stull \cite{LutStu18}.  In short, the point-to-set principles have allowed for information-theoretic techniques to be imported into the study of fractal geometry to great effect.

The present study contributes to work in this area in several respects.  First, we introduce the machinery of unequal costs coding into the study of effective notions of fractal dimension, a new development in this area.  This study can thus be seen as providing a proof of concept of the utility of unequal costs coding in a relatively constrained setting, i.e., in the context of self-similar fractal trees, obtaining new results on the effective dimension of points in such trees in a way that makes fundamental use of aspects of unequal costs coding (namely, the relationship between the similarity dimension of a self-similar fractal and the channel capacity of an unequal costs code, as laid out in Theorem \ref{thm-dim-cc}), providing a more efficient approach than using previously developed techniques in this area.  

Second, in the constrained setting of self-similar fractal trees, we unearth new connections between the structure of such trees, the properties of the associated unequal cost codes, and certain symbolic dynamical features of directed graphs derived from these trees, connections that do not readily generalize to a broader setting in any obvious way.  In addition, this connection relies on a new result on the characteristic polynomial of adjacency matrices of the above-mentioned directed graphs, which we establish using ideas from spectral graph theory, an area with which the literature on algorithmic information theory has engaged very little. Moreover, previous work on the overlap between symbolic dynamics and algorithmic information theory does not explicitly discuss the significance of the Parry measure in this context, whereas the Parry measure features prominently in the discussion in Section \ref{sec-dynamics}.

Finally, we anticipate that the machinery of unequal costs coding will have a wider range of applicability in the study of effective fractal dimensions.  For instance, although the point-to-set principles are not needed to understand the structure of the self-similar fractal trees we consider here (as they are essentially finitary objects, being recursively generated from finite trees), if one were to consider a natural extension of self-similar fractal trees that involve a more general type of length function (for instance, one in which the cost of a symbol may vary depending on its location in a given sequence), this might allow one to obtain similar results with more general self-similar systems using sophisticated tools such as the point-to-set principles.  

With this context in mind, we now turn to the main results of this study.

\section{Coding with unequal costs and self-similar trees}\label{sec-coding}

\subsection{Unequal costs coding}\label{subsec-ucc}
Let us consider the finite alphabet $\Sigma_k$ for some fixed $k\geq 2$.  As defined in the information theory literature on coding with unequal costs, we define a length function to simply be a function $\ell:\Sigma_k\rightarrow\mathbb{N}$; see, for instance, \cite{Var71, Cha86, Abr94, Abr97, GolKenYou02}.  Note that we could consider a broader notion of length functions, for instance, length functions that are real-valued, length functions that are not defined merely in terms of individual symbols but may give a symbol varying cost depending on where it is located in a given string or depending on the symbols that precede it, and so on.  However, for our purposes, we only need the simple notion of a length function as defined above.   

A key notion in the study of coding with unequal letter costs that will be useful in this study is that of a channel capacity \cite{BeaBerMar10}, \cite[Chapter 4]{CsiKor11}.  For a length function $\ell$ defined on alphabet $\Sigma_k$, the \emph{channel capacity} associated to $\ell$ is the unique $\alpha\in\mathbb{R}^{\geq 0}$  that satisfies the expression
\begin{equation}\tag{$\dagger$}\label{eq1}
\sum_{i=0}^{k-1}2^{-\alpha\ell(i)}=1.
\end{equation}
Informally, the channel capacity of a length function measures the average amount of information transmitted per unit cost (see the discussion in \cite[Chapter 4]{CsiKor11}). It is not difficult to see that the channel capacity associated to $\ell$ is unique.  Indeed, if we set $r =2^{-\alpha}$, we can rewrite Equation (\ref{eq1}) as
\begin{equation*}
\sum_{i=0}^{k-1}r^{\ell(i)}=1.
\end{equation*}
As the equation $f(r)=\sum_{i=0}^{k-1}r^{\ell(i)}-1$ is strictly increasing on [0,1] with $f(0)<0$ and $f(1)>0$, uniqueness follows.  Observe more generally that the channel capacity $\alpha$ associated to a length function $\ell$ on $\Sigma_k$ allows use to define a Bernoulli measure $\mu_\ell$ on $\Sigma_k^\infty$ induced by setting, for each $i\in\Sigma_k$,
\[
\mu_\ell(i)=r^{\ell(i)},
\]
where $-\log(r)=\alpha$.  We will refer to $\mu_\ell$ as the \emph{measure derived from $\ell$}.

\subsection{Self-similar fractal trees}
We now turn to defining the similarity dimension of a self-similar fractal generated by a finite tree.  Let $T\subseteq\Sigma_m^*$ be a finite tree for some $m\in\mathbb{N}$.  Suppose that $T$ has $k$ terminal nodes $\tau_0,\dotsc,\tau_{k-1}$. 
We define an infinite, self-similar tree $T^*$ from $T$ simply by concatenating a copy of $T$ at each of its terminal nodes and then repeating this process recursively.  

Let us define $F_T=[T^*]$.  Clearly, $F_T$ is a self-similar fractal, as above each terminal node of $T$, we place a full copy of $F_T$; hereafter, we will refer to $F_T$ as the \emph{self-similar fractal tree generated by $T$}.  

Since each such copy of $F_T$ begins at some level of the tree deeper than the root, this amounts to scaling that copy.  In our context, this scaling factor is precisely the ratio of the size of the copy of $F_T$ to the size of $F_T$, which is precisely the similarity ratio from fractal geometry (see \cite[Section 9.2]{Fal04}).  Here, size is given by the standard metric $d$ on $\Sigma_m^\infty$ according to which two sequences $x,y\in\Sigma_m^\infty$ that agree on the first $j$ symbols but disagree on the $(j+1)$-st satisfy $d(x,y)=m^{-j}$.

Let us calculate the similarity ratios for these copies of $F_T$.  Given a terminal node $\tau$ in $T$, by concatenating a copy of $F_T$ at the end of $\tau$, this amounts to scaling $F$ by a factor of $m^{-|\tau|}$.  Thus, the similarity ratios corresponding to the terminal nodes $\tau_0,\dotsc,\tau_{k-1}$ are equal to $m^{-|\tau_0|},\dotsc,m^{-|\tau_{k-1}|}$, respectively.  Then the similarity dimension $\sdim(F)$ of $F$ is the unique real number $\beta$ satisfying
\begin{equation}\tag{$\dagger\dagger$}\label{eq2}
\sum_{i=0}^{k-1}m^{-\beta|\tau_i|}=1.
\end{equation}

\subsection{Relating length functions and self-similarity}
We now spell out how coding with unequal costs is related to the similarity dimension of a self-similar tree.

Given $T$ as above with terminal nodes $\tau_0,\dotsc,\tau_{k-1}$, we associate to each terminal node $\tau_i$ the symbol $i\in\Sigma_k$ and define a length function $\ell_T$ satisfying
\[
\ell_T(i)=|\tau_i|
\]
for each $i\in\{0,\dotsc,k-1\}$.  Hereafter, let us refer to $\ell_T$ as the \emph{length function induced by $T$}.  We then have the following.

\begin{thm}\label{thm-dim-cc}
Let $T\subseteq\Sigma_m^*$ be a finite tree for some $m\in\mathbb{N}$. Let $\ell_T$ be the length function induced by $T$ and $F_T$ the self-similar fractal generated by $T$.  Then the channel capacity of $\ell_T$ is equal to $\log(m)\cdot\sdim(F)$. In particular, in the case that $m=2$, the channel capacity of $\ell_T$ is equal to $\sdim(F_T)$.
\end{thm}

\begin{proof}
From Equation (\ref{eq1}), the channel capacity $\alpha$ of $\ell_T$ satisfies
\[
\sum_{i=0}^{k-1}2^{-\alpha\ell_T(i)}=1.
\]
Since $\ell_T(i)=|\tau_i|$ for $i\in\Sigma_k$, changing the base of the exponential from 2 to $m$, we can rewrite above equation as
\[
\sum_{i=0}^{k-1}m^{-\frac{\alpha}{\log(m)}|\tau_i|}=1.
\]
But from Equation (\ref{eq2}), $\log(m)\cdot\sdim(F_T)$ is the unique solution to the above equation, and hence we can conclude that $\log(m)\cdot\sdim(F_T)=\alpha$.
\end{proof}

\subsection{Deriving the Lutz/Mayordomo formulas for self-similar trees}

The observation that the similarity dimension of a self-similar $F\subseteq \Sigma_m^\infty$ generated by a finite tree $T\subseteq \Sigma_m^*$ is equal to the channel capacity of the unequal costs code determined by $T$ allows us to derive analogues of the Lutz/Mayordomo formulas for the effective dimension of points in self-similar fractals as stated in the introduction by means of a relatively straightforward proof.

First, we need to define one additional notion.  Given a finite tree $T\subseteq\Sigma_m^*$ with $k$ terminal nodes $S=\{\tau_0,\dotsc,\tau_{k-1}\}$ for some $m,k\in\N$ and the corresponding self-similar fractal $F_T$, we define a one-to-one correspondence $\Psi$ between elements of $F_T$ and elements of $\Sigma_k^\infty$ as follows. First we define a coding map $\psi:S\rightarrow\Sigma_k$ by setting $\psi(\tau_i)=i$ for $i\in\Sigma_k$.  We extend $\psi$ to a functional $\Psi: F_T\rightarrow\Sigma_k^\infty$ as follows.  Given $x\in F_T$, there is a sequence $(n_i)_{i\in\N}\subseteq\N$ with $n_0=0$ such that for each $i\in\N$, $x\uh[n_i,n_{i+1})$ is a terminal node of $T$; let us call this the \emph{$T$-sequence of $x$}.  Note that since $S$ is prefix-free, the $T$-sequence of each $x\in F_T$ is unique. We then define $\Psi_T$ by setting
\begin{multline*}
\Psi_T(x)=\psi(x\uh [n_0,n_1))^\frown\psi(x\uh [n_1,n_2))^\frown\cdots \\
^\frown\psi(x\uh[n_i,n_{i+1}))^\frown\cdots
\end{multline*}
Setting $y=\Psi_T(x)$, we will hereafter refer to $y$ as the \emph{coding sequence of $x$}. Note that we can similarly extend $\psi$ to a function $\hat\psi:S^*\rightarrow\Sigma_k^*$, which we will use shortly.  We can now state the analogue of the Lutz/Mayordomo formulas in our context.

\begin{thm}\label{thm-lm}
Let $T\subseteq\Sigma_m^*$ be a finite tree with $k$ terminal nodes for some $m,k\in\N$ with corresponding length function $\ell=\ell_T$, self-similar fractal $F_T$, and the measure $\mu_\ell$ derived from $\ell$.  Then for $x\in F_T$,
\[
\dim(x)=\sdim(F_T)\dim^{\mu_{\ell}}(y) 
\]
and
\[
\ddim(x)=\sdim(F_T)\ddim^{\mu_{\ell}}(y),
\]
where $y\in\Sigma_k^\infty$ is the coding sequence of $x\in F_T$.
\end{thm}

\begin{proof}

Let $S=\{\tau_0,\dotsc,\tau_k\}$ be the set of terminal nodes of $T$, and let $\hat\psi:S^*\rightarrow\Sigma_k^*$ be the extension of $\psi$ as discussed above.  Note that $\ell(i)=|\tau_i|$ for $i\in\Sigma_k$. Given $x\in F_T$, let $(n_i)_{i\in\N}\subseteq\N$ be the $T$-sequence of $x$ and $y\in\Sigma_k^\infty$ the coding sequence for $x$.  We first establish the following claim:

\medskip
\noindent \emph{Claim:} For $j\in \N$, $K(x\uh {n_j})=K(y\uh j)+O(1)$.

\medskip

The argument to establish this claim is fairly routine, but we include it here for the sake of completeness. First note that $\hat\psi(x\uh n_j)=y\uh j$.   Indeed, since the $T$-sequence of $x$ is unique, we calculate the value $\hat\psi(x\uh n_j)$ first by decomposing $x\uh n_j$ into the concatenation of $j$ elements of $S$.  Applying $\psi$ to each of these in turn and concatenating the outputs yields the string $y\uh j$.  

Next we define a machine $M:\Sigma_2^*\rightarrow\Sigma_m^*$ as follows.  For $\sigma\in\Sigma_2^*$, if $U(\sigma)\halts$ and $U(\sigma)$ can be written as the concatenation of members of $S$ (so that $\hat\psi(U(\sigma))$ is defined), then we set $M(\sigma)=\hat\psi(U(\sigma))$.  In particular,   $U(\sigma)\halts=x\uh n_j$ if and only if $M(\sigma)\halts=y\uh j$.  By the optimality of $K$, we have $K(y\uh j)\leq K_M(y \uh j)+O(1)\leq K(x\uh n_j)+O(1)$.  Similarly, as $\hat\psi^{-1}$ is invertible, we can use $\hat\psi^{-1}$ we can establish that $K(x\uh n_j)\leq K(y\uh j)+O(1)$, and so we have established the claim.

There exists a sequence $\tau_{i_0},\dotsc,\tau_{i_{j-1}}$ of elements of $S$ such that $x\uh n_j={\tau_{i_0}}^\frown \dotsc^\frown\tau_{i_{j-1}}$.  Then if $-\log(r)$ is the channel capacity associated to $\ell$,
\begin{equation}\label{eq-length}
\begin{split}
n_j=\sum_{e=0}^{j-1}|\tau_{i_e}|&=\sum_{e=0}^{j-1}\ell(i_e)=\ell(y\uh j)\\
&=\left(\frac{1}{-\log(r)}\right)\left(-\log\Bigl(r^{\ell(y\uh j)}\Bigr)\right)\\
&=\left(\frac{1}{-\log(r)}\right)\left(-\log(\mu_\ell(y\uh j))\right).
\end{split}
\end{equation}
Then we have
\begin{align*}
\frac{K(x\uh {n_j})}{n_j\log(m)}&=\frac{K(y\uh j)+O(1)}{n_j\log(m)}\\
&=\left(\frac{-\log(r)}{\log(m)}\right)\frac{K(y\uh j)+O(1)}{-\log(\mu_\ell(y\uh j))}\\
&=\sdim(F_T)\frac{K(y\uh j)+O(1)}{-\log(\mu_\ell(y\uh j))},
\end{align*}
where the first equality follows from the above claim, the second equality follows from Equation (\ref{eq-length}), and the third equality follows from the fact that  $-\log(r)$ is the channel capacity of $\ell$ and $-\log(r)=\log(m)\cdot\sdim(F_T)$ by Theorem \ref{thm-lm}.
Taking the limit infimum and limit supremum of both sides of this equality yields 
\begin{equation}\label{eq-a}
\liminf_{j\rightarrow\infty}\frac{K(x\uh {n_j})}{n_j\log(m)}=\sdim(F_T)\dim^{\mu_\ell}(y)
\end{equation}
and
\begin{equation}\label{eq-b}
\limsup_{j\rightarrow\infty}\frac{K(x\uh {n_j})}{n_j\log(m)}=\sdim(F_T)\ddim^{\mu_\ell}(y).
\end{equation}

\noindent Moreover, we have
\[
\dim(x)=\liminf_{n\rightarrow\infty}\frac{K(x\uh n)}{n\log(m)}\leq\liminf_{j\rightarrow\infty}\frac{K(x\uh {n_j})}{n_j\log(m)}
\]
and
\[
\limsup_{j\rightarrow\infty}\frac{K(x\uh {n_j})}{n_j\log(m)}\leq\limsup_{n\rightarrow\infty}\frac{K(x\uh n)}{n\log(m)}=\ddim(x),
\]
from which, combined with Equations (\ref{eq-a}) and (\ref{eq-b}), we conclude 
\begin{equation}\label{eq-c}
\dim(x)\leq \sdim(F_T)\dim^{\mu_\ell}(y)
\end{equation}
and
\begin{equation}\label{eq-d}
\sdim(F_T)\dim^{\mu_\ell}(y)\leq \ddim(x).
\end{equation}

To conclude the proof, we have to consider the values $\frac{K(x\uh n)}{n}$ for $n\in (n_j,n_{j+1})$ for $j\in\N$.  We use the subadditivity of $K$:  there is some $c\in\N$ such that for all $\sigma,\tau\in \Sigma_k^*$,  $K(\sigma^\frown\tau)\leq K(\sigma)+K(\tau)+c$. Let $S_0$ consist of all proper suffixes of the terminal nodes of $T$, and we let $s_0=\max\{K(\sigma)\colon\sigma\in S_0\}$.  Given $n\in (n_j,n_{j+1})$, we can write $n=n_j+i$ for some $i< n_{j+1}-n_j$.  Then
\begin{align*}
&\frac{K(x\uh n_{j+1})}{n_{j+1}}\leq\\
&\frac{K(x\uh (n_j+i))+K(x\uh[n_j+i,n_{j+1}))+c}{n_{j+1}}\leq\\
 &\frac{K(x\uh (n_j+i))+c+s_0}{n_{j+1}},
\end{align*}
where the first inequality follows from the subadditivity of $K$ and the second from the fact that $x\uh[n_j+i,n_{j+1}))\in S_0$.
Rearranging yields
\begin{align*}
\frac{K(x\uh n_{j+1})-O(1)}{n_{j+1}}&\leq\frac{K(x\uh (n_j+i))}{n_{j+1}}\\
& \leq \frac{K(x\uh (n_j+i))}{n_j+i}=\frac{K(x\uh n)}{n}.
\end{align*}
Thus, for all $n$ such that $n\neq n_j$ for all $j$, there is some $j$ such that $n\in(n_j,n_{j+1})$ for which 
\[
\frac{K(x\uh n_j)-O(1)}{n_j}\leq\frac{K(x\uh n)}{n}.
\]
Applying the limit infimum of both sides and applying Equation (\ref{eq-a}), we get
\begin{equation*}
\sdim(F_T)\dim^{\mu_\ell}(y)\leq \dim(x),
\end{equation*}
which, combined with Equation (\ref{eq-c}), gives us
\[
\dim(x)=\sdim(F_T)\dim^{\mu_\ell}(y).
\]

Similarly, let $S_1$ consist of all proper initial segments of the terminal nodes of $T$.  Then we let $s_1=\max\{K(\sigma)\colon\sigma\in S_1\}$.  Given $n\in (n_j,n_{j+1})$ with $n=n_j+i$ for some $i< n_{j+1}-n_j$ as above, we have
\begin{align*}
\frac{K(x\uh n)}{n}&=\frac{K(x\uh (n_j +i))}{n_j+i}\\
&\leq \frac{K(x\uh n_j)+K(x\uh[n_j,n_j+i))+ c}{n_j+i}\\
&\leq \frac{K(x\uh n_j)+c+s_1}{n_j+i}\\
&\leq \frac{K(x\uh n_j)+O(1)}{n_j}
\end{align*}
where the first inequality follows from the subadditivity of $K$ and the second from the fact that $x\uh[n_j,n_j+i)\in S_1$.  Thus, for all $n$ such that $n\neq n_j$ for all $j$, there is some $j$ such that $n\in(n_j,n_{j+1})$ for which 
\[
\frac{K(x\uh n)}{n}\leq \frac{K(x\uh n_j)+ O(1)}{n_j}.
\]
Thus 
\begin{align*}
\ddim(x)&=\limsup_{n\rightarrow\infty}\frac{K(x\uh n)}{n\log(m)}\\
&\leq \limsup_{n\rightarrow\infty}\frac{K(x\uh n_j)+O(1)}{n_j\log(m)}
\end{align*}
which, combined with Equation (\ref{eq-d}), gives us
\[
\ddim(x)=\sdim(F_T)\ddim^{\mu_\ell}(y).
\]

\end{proof}

\section{A dynamical systems perspective on self-similar fractal trees}\label{sec-dynamics}

\subsection{Self-similar trees and path sets} We next take a dynamical systems perspective on self-similar fractal trees generated by a finite tree.  In this context, such a self-similar tree can be seen as coding the set of one-sided infinite walks through a certain directed graph, where every walk begins at a distinguished node, providing an instance of what has been referred to as a \emph{path set} by Abrams and Lagarias in \cite{AbrLag14}.  

Fix $n\in\N$. A \emph{pointed graph} $(\G, v)$ over $\Sigma_n$ consists of an edge-labeled finite directed graph $\G = (G,\E)$, where $G=(V,E)$ is a directed graph (in which loops and multiple edges are permitted) with a distinguished vertex $v\in V$ and $\E\subseteq E\times \Sigma_n$ (i.e., the edges of $G$ are labeled with elements of $\Sigma_n$).  We further require that  the labels on different edges between any two fixed vertices must be distinct.  We then define the \emph{path set} $P=X_\G(v)\subseteq \Sigma_n^\infty$ given by the pointed graph $\G$ to be the set of infinite sequences obtained from the edge labels of all possible one-sided infinite walks in $G$ that begin at the distinguished vertex $v$.

Path sets are more general than the sets of paths through self-similar fractal trees that we are considering here, but we can characterize the fractals under consideration in terms of a restricted class of path sets.  First, note that a path set can be obtained from different pointed graphs; we shall refer to an underlying pointed graph for a path set $P$ as a \emph{presentation} of $P$. A directed graph $G$ with edge labels $\E$ is \emph{right-resolving} if for every vertex $v$ in $G$, all directed edges starting from $v$ have distinct edge labels.  In addition, a directed graph $G$  is \emph{irreducible} (or \emph{strongly connected}) if for every pair of vertices $u$ and $v$ in $G$, there is a directed walk in $G$ from $u$ to $v$.  It can be shown that every path set has a right-resolving presentation, but not every path set has an irreducible presentation (see \cite{AbrLag14}).  

\begin{thm}\label{thm-treegraph}
Fix $m\in\N$.
\begin{itemize}
\item[(i)] Every self-similar fractal tree $F$ generated by a finite tree $T\subseteq \Sigma_m^*$ can be obtained as a path set with a right-resolving, irreducible presentation.
\item[(ii)] Every path set over $\Sigma_m$  with a right-resolving, irreducible presentation in which (a) every multiple edge terminates in the distinguished vertex and (b) every cycle contains the distinguished vertex is a self-similar fractal tree generated by some finite tree $T$.
\end{itemize}
\end{thm}

\begin{proof}
(i) First we show how to associate a pointed directed graph $\mathcal{G}$ with a given finite tree $T$.  First let $n$ be the number of non-terminal nodes in $T$, written $\sigma_1,\dotsc,\sigma_n$ in length-lexicographic order, so that $\sigma_1$ is the root of the tree (we begin counting at 1 instead of 0 to align with conventions from linear algebra that we will use shortly).  For each such node $\sigma_i$, we will have a corresponding vertex $v_i$ in $G$; the vertex $v_1$ corresponding to $\sigma_1$ will be the distinguished vertex.  Moreover, for each pair of nodes $\sigma_i$ and $\sigma_j$ such that $\sigma_i$ is a parent of $\sigma_j$, we include in $G$ a directed edge from $v_i$ to $v_j$ with the same label on the edge from $\sigma_i$ to $\sigma_j$.  Lastly, for each node $\sigma_i$ that is the parent of a terminal node, we will include in $G$ one directed edge from $v_i$ to $v_1$ for each terminal node connected to $\sigma_i$, labeling the directed edge with the same label on the edge connecting $\sigma_i$ to the corresponding terminal node in $T$.  It is not difficult to verify that the set of infinite paths $F$ through $T^*$ is equal to the set of infinite one-way walks through $G_T$ that begin at $v_1$.

\bigskip

\noindent (ii) Suppose $\G=(G,v)$ is a right-resolving, irreducible directed graph satisfying the conditions (a) and (b) as above.  The construction of the tree $T$ is straightfoward.  Let $u$ be a vertex for which there is a directed edge that has the distinguished vertex $v$ as its terminal vertex. For each such directed edge, we disconnect it from $v$ and connect it to a new vertex $w$ that is not connected to any other vertex in the graph, keeping the same edge label as in the original directed graph (note that there may be multiple new vertices connected to $u$ to which we perform this operation).  Repeating this process for all such vertices $u$ yields a new pointed graph $\G'=(G',v)$. We claim that the underlying graph of $G'$ is a tree with root $v$.  

First, note that there are no multiple edges in $G'$, since by condition (a), the only multiple edges in $G$ have been removed and replaced with new edges terminating in distinct vertices.  Second, we claim that $G'$ is acyclic.  Indeed, if there were a cycle in $G'$, then it would be a cycle in $G$ (since no cycles are created in transforming $G$ to $G'$).  By condition (b), this cycle contains the distinguished vertex $v$.  But this contradicts the fact that no directed edge terminating in $v$ in $G$ is contained in $G'$ by our construction.  Finally, since $G$ is irreducible, $G'$ is connected (but not strongly connected), since the transformation from $G$ to $G'$ only results in the removing of cycles, keeping all other edge relations intact.  If we take $T$ to be the underlying graph of $G'$ (removing the orientation of the directed edges), this is the desired tree $T$.  As in the proof of (i), it is routine to verify that the set of infinite one-way walks through $\G$ beginning at $v$ is equal to the set of infinite paths $F$ through $T^*$.
\end{proof}

See Figure \ref{img1} for an example of a finite tree and its corresponding pointed directed graph.

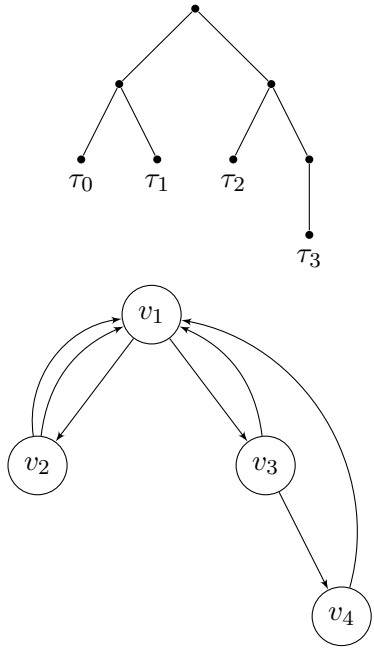
\begin{figure}[h]

\begin{center}

\tikzstyle{level 1}=[level distance=1cm, sibling distance=2cm]
\tikzstyle{level 2}=[level distance=1cm, sibling distance=1cm]

\tikzstyle{nodot} = [text width=4em, text centered]
\tikzstyle{dotted} = [circle, minimum width=3pt,fill, inner sep=0pt]
\begin{tikzpicture}[grow=down, sloped]
\node[dotted] {}
    child {
        node[dotted] {}        
        child {
            node[dotted,label=below:{$\tau_0$}] {}
            edge from parent
            node[above] {}
            node[below]  {}
        }
        child {
            node[dotted,label=below:{$\tau_1$}] {}
            edge from parent
            node[above] {}
            node[below]  {}
        }
        edge from parent 
        node[above] {}
        node[below]  {}
}
child {
    node[dotted] {}        
    child {
            node[dotted, label=below:{$\tau_2$}] {}
            edge from parent
            node[above] {}
            node[below]  {}
        }
        child {
            node[dotted] {}
                         child {
            node[dotted, label=below:{$\tau_3$}] {}
            edge from parent 
            node[above] {}
            node[below]  {}
            }
                edge from parent
                node[above] {}
                node[below]  {}
            }
        edge from parent         
            node[above] {}
            node[below]  {}
    };
\end{tikzpicture}

  \end{center}

\begin{center}
  \begin{tikzpicture}

\tikzset{vertex/.style = {shape=circle,draw,minimum size=1.5em}}
\tikzset{edge/.style = {->,> = latex'}}
\node[vertex] (a) at  (0,0) {$v_2$};
\node[vertex] (b) at  (1.5,2) {$v_1$};
\node[vertex] (c) at  (3,0) {$v_3$};
\node[vertex] (d) at  (4,-2) {$v_4$};
edges
\draw[edge] (b) to (a);
\draw[edge] (b) to (c);
\draw[edge] (c) to (d);
\draw[edge] (a)  to[bend left] (b);
\draw[edge] (a) to[bend left = 45] (b);

\draw[edge] (c) to[bend right] (b);

\draw[edge] (d)  to[bend right=47.5] (b);

%
%
%
\end{tikzpicture}
  \end{center}

\caption{A finite tree $T$ and its corresponding pointed directed graph $G$}
  \label{img1}

\end{figure}

\subsection{Spectral analysis of the pointed graph associated with a finite tree}\label{subsec-spectral}

The relationship between self-similar fractal trees and path sets bears further fruit for us, particularly when we study the measure $\mu_\ell$ associated with the length function $\ell_T$ derived from a finite tree $T$.  In particular, for a given finite tree $T$ and associated directed graph $G$, we can derive the Perron eigenvalue associated with $G$ (which we define below) solely in terms of properties of the tree $T$.  From this Perron eigenvalue, we can derive the so-called Parry measure on the set of infinite walks through $G$ that begin at any vertex.  We will further show the measure $\mu_\ell$ of the collection of coding sequences for members of $F_T$ can be derived from this Parry measure.
 
Fix $m\in\N$.  Given a finite binary tree $T\subseteq\Sigma_m^*$ with $n$ non-terminal nodes, let $\mathcal{G}=(G,\mathcal{E})$ be the corresponding pointed directed graph on $n$ vertices $v_1,\dotsc,v_n$ as in the proof of Theorem \ref{thm-treegraph}(i).  Let $A$ be the adjacency matrix of  $G$, i.e., the $n\times n$ matrix where for $1\leq i,j\leq n$, $a_{i,j}$ is the number of directed edges from vertex $v_i$ to $v_j$ in $G$.  Since the entries of $A$ are non-negative and the graph $G$ is irreducible (so that the matrix $A$ is irreducible), by the Perron-Frobenius theorem (see, e.g., \cite[Theorem 4.2.3]{LinMar21}), $A$ has a largest positive real eigenvalue $\rho$ (the so-called \emph{spectral radius} of $A$). As noted in the introduction, we can show that $\log(\rho)$ is equal to the channel capacity of $\ell_T$, the length function induced by $T$, which we obtain using the following theorem.

\begin{thm}\label{thm-spectral}
Given a finite tree $T$ with $n$ non-terminal nodes, let $A$ be the $n\times n$ adjacency matrix of the directed graph associated to $T$, and let
$p_A(z)=z^n+c_1z^{n-1}+\dotsc+c_{n-1}z+c_n$ be the characteristic polynomial of $A$.  Then for $i=1,\dotsc,n$,
$-c_i$ is equal to the number of terminal nodes in $T$ of depth $i$.
\end{thm}

\begin{proof}
We proceed by induction on $k\leq n$.  Let $G$ be the directed graph derived from $T$ as in the proof of Theorem \ref{thm-treegraph}(i). For $k=1$, it is a standard fact in linear algebra that for the characteristic polynomial associated with a matrix $A$, we have $c_1 = -\mathrm{trace}(A)$.  Note that the trace of $A$ counts the number of self-loops in $G$.  By the construction of $G$ from $T$, the only vertex of $G$ that can have a self-loop is the vertex $v_1$, where each such self-loop in $G$ corresponds to a terminal node in $T$ of length $1$.  Thus, $-c_1$ is the number of terminal nodes of $T$ of depth 1.

For a fixed $k$ with $2\leq k\leq n$, assume that we have established the result for all $j<k$.  We will use the Leverrier's algorithm for calculating the coefficients of the characteristic polynomial of an $n\times n$ matrix $A$ (see, e.g., \cite{Gan59}), according to which for $2\leq k<n$, 
\[
c_k = -\frac{1}{k}\mathrm{trace}(A^k+c_1A^{k-1}+\dotsc+c_{k-1}A).
\]
We begin this case by noting that the number of terminal nodes of $T$ of depth $k$ is equal to the number of $k$-cycles in $G$ (that is, cycles of length $k$) that begin at $v_1$ and only pass through $v_1$ once.  We further observe that the entries along the diagonal of $A^k$, each of which contributes to the trace of $A^k$, count the total number of $k$-cycles in $G$, where this number is the sum of
\begin{itemize}
\item[(a)] the number of $k$-cycles in $G$ that begin at $v_1$ and only pass through $v_1$ once,
\item[(b)] the number of $k$-cycles in $G$ that begin at some $v\neq v_1$ and only pass through $v_1$ once, and
\item[(c)] the number of $k$-cycles in $G$ that pass through $v_1$ more than once (regardless of where they start).
\end{itemize}
As the number of cycles counted in (a) and (b) above is equal to $k$ times the number of terminal nodes of $T$ of depth $k$, each cycle counted in (a) is counted $k-1$ additional times in (b) (corresponding to the $k-1$ different vertices other than $v_1$ at which the each such cycle can start).

Next, for any $j$ with $1\leq j\leq k-1$, by our inductive hypothesis, we have

\begin{multline*}
-\mathrm{trace}(c_jA^{k-j})=\\
(\text{$\#$ of terminal nodes in  $T$ of depth $j$})\\
 \times (\text{$\#$ of $(k-j)$-cycles in $G$})\\
=-(\text{$\#$ of $j$-cycles in $G$ satisfying (a)})\\
\times (\text{$\#$ of $(k-j)$-cycles in $G$}).\hspace{.55in}
\end{multline*}
This latter expression counts all $k$-cycles in $G$ that have one of the following two forms:  
\begin{itemize}
\item[(i)] if the cycle begins at $v_1$, the cycle consists of a $j$-cycle that passes through $v_1$ once followed by a $(k-j)$-cycle (with no constraints on the number of times it passes through $v_1$), or
\item[(ii)] if the cycle does not begin at $v_1$, it begins with a $(k-j)$-cycle until it reaches $v_1$, completes a $j$-cycle that begins and ends at $v_1$ (without passing through $v_1$ any other time), and then completes the rest of the original $(k-j)$-cycle.
\end{itemize}
If we take the sum of all such terms $-\mathrm{trace}(c_jA^{j-k})$ for $1\leq j\leq k-1$, this will count all $k$-cycles that fall under the condition (c) given above.

It thus follows that the term 
\begin{multline}\label{eq-trace}
\mathrm{trace}(A^k+c_1A^{k-1}+\dotsc+c_{k-1}A)=\\
\mathrm{trace}(A^k)-(-\mathrm{trace}(c_1A^{k-1})+\dotsc\\
-(-\mathrm{trace}(c_{k-1}A))
\end{multline}
counts the number of $k$-cycles that fall under conditions (a) and (b) above, which as noted above, equals $k$ times the number of terminal nodes of $T$ of depth $k$.  Lastly, dividing the expression given above by (\ref{eq-trace}) by $k$ and using Leverrier's Method, we conclude that $-c_k$  is equal to the number of terminal nodes of $T$ of depth $k$.  This establishes the inductive step, and so the result holds.

\end{proof}

\begin{thm}\label{thm-Perron-cc}
Let $T$ be a finite tree that induces a length function $\ell_T$, and let $\rho$ be the Perron eigenvalue of the adjacency matrix of the directed graph associated to $T$.  Then $\log(\rho)$ is the channel capacity of $\ell_T$.

\end{thm}

\begin{proof}
Let $A$ be the adjacency matrix of the directed graph associated to $T$, and let $p_A(z)=z^n+c_1z^{n-1}+\dotsc+c_{n-1}z+c_n$ be the characteristic polynomial of $A$.  If $\rho\geq 1$ is the Perron eigenvalue of $p_A$, we have
\[
\rho^n+c_1\rho^{n-1}+\dotsc+c_{n-1}\rho+c_n=0.
\]
Dividing through by $\rho^n$ and rearranging yields
\[
-c_1\left(\frac{1}{\rho}\right)-\dotsc-c_{n-1}\left(\frac{1}{\rho}\right)^{n-1}-c_n\left(\frac{1}{\rho}\right)^n=1.
\]
Setting $\alpha=\log(\rho)$, so that $2^{-\alpha}=\frac{1}{\rho}$, by substitution we have
\begin{equation}\label{eq-alpha}
-c_12^{-\alpha}-\dotsc-c_{n-1}2^{-\alpha(n-1)}-c_n2^{-\alpha n}=1.
\end{equation}
By Theorem \ref{thm-spectral},
for $i=1,\dotsc,n$, $-c_i$ is equal to the number of terminal nodes in $T$ of depth $i$, so we can rewrite Equation (\ref{eq-alpha}) as
\begin{equation}\label{eq-alpha2}
\sum_{\tau\in S}2^{-\alpha|\tau|}=1,
\end{equation}
where $S$ is the set of terminal nodes of $T$.  As the unique positive solution of Equation (\ref{eq-alpha2}) is the channel capacity of $\ell$, the conclusion follows.

\end{proof}

\subsection{Restricting the Parry measure on a sofic shift}

In their work on path sets, Abram and Lagarias showed that path sets are a generalization of sofic shifts.  A \emph{sofic shift} is subset of $\Sigma_n^\infty$ for some $n\in\N$  that consists of all sequences of symbols obtained from an one-sided infinite walks though a finite directed graph with edge labels, where, unlike the case with path sets, the walks can start from any vertex of the graph.  Abram and Lagarias further proved the equality of several notions of topological entropy relating path sets and sofic shifts.  We review the definitions.  First, for a path set $P$, let $N_n(P)$ denote the number of distinct blocks of length $n$ occurring in any member of $P$.  Similarly, $N_n^I(P)$ denotes the number of distinct initial blocks of length $n$ from symbols in $P$.  Then we have the following definitions.

\begin{dfn}{\ }
\begin{enumerate}
\item The \emph{path topological entropy} of a path set $P$ is defined to be 
\[
H_p(P) = \limsup_{n\rightarrow\infty} \frac{1}{n}\log N_n(P).
\]
\item The \emph{topological entropy} of a path set $P$ is
\[
H_{\mathit{top}}(P) = \limsup_{n\rightarrow\infty} \frac{1}{n}\log N_n^I (P).
\]
\end{enumerate}
\end{dfn}

Finally, given a path set $P\subseteq \Sigma_n^\infty$, the \emph{one-sided shift closure} $\overline P$ of $P$ is $\bigcup_{i\in\N}S^i(P)$, where $S$ is the shift operator on $\Sigma_n^\infty$.  Note that for a path set $P$, its shift closure $\overline{P}$ is a sofic shift, for which the notion of topological entropy $H_{\mathit{top}}(\overline{P})$ is well-defined.  Then we have:

\begin{thm}[Abram, Lagarias \cite{AbrLag14}]
Let $P\subseteq\Sigma_n^\infty$ be a path shift.  Then $H_p(P)=H_{\mathit{top}}(P)=H_{\mathit{top}}(\overline{P})$.
\end{thm}

We apply this result in our own context as follows.  Given a finite tree $T$ and its associated pointed directed graph $\G=(G,v_1)$, we can assume by Theorem \ref{thm-treegraph}(i) that $G$ is right-solving and irreducible.  As noted above, the one-sided shift closure of  the set of sequences obtained from one-sided infinite walks through $G$ is a sofic shift, which we will write as $X_G$.  In this context, $G$, equipped with its edge labels, is referred to as a \emph{presentation} of $X_G$.  Moreover, since the cycles through $G$ that begin and end at the distinguished vertex $v_1\in G$ correspond to the terminal nodes of $T$, it is not difficult to establish that $G$ is \emph{minimal}, in the sense that there is no presentation $G'$ of $X_G$ that has strictly fewer vertices than $G$.  Finally, as is known in the symbolic dynamics literature, a sofic shift given by an irreducible, right-resolving, minimal presentation has a unique measure of maximal entropy, which is known as the \emph{Parry measure}. 

We should note that the entropy of a measure on a graph is slightly more general than the entropy of a measure on a sequence space such as $\Sigma_n^\infty$.  For a measure $\mu$ on a graph, we first specify initial probabilities for the initial vertex of each edge, written $i(e)$ for an edge $e$, and then we specific conditional probabilities $\mu(e\mid i(e))$ for each edge of the graph.  Thus, the entropy of $\mu$ is
\[
h(\mu)=-\sum_{e\in E(G)}\mu(i(e))\log(\mu(e\mid i(e))\mu(e\mid i(e)).
\]
Finally, the measure $\mu$ of maximal entropy on a sofic shift $X$ is one that attains the value $h(\mu)=H_{\mathit{top}}(X)$.

For the case of $X_G$ defined above, let us look more closely at the derivation of the Parry measure on $X_G$, following the presentation found in see \cite[Section 13.3]{LinMar21}.   Assuming that $T$ contains exactly $n$ non-terminal nodes, so that $G$ consists of vertices $v_1,\dotsc,v_n$, let $A$ be the $n\times n$ adjacency matrix of $G$ as in Section \ref{subsec-spectral} above.  Let $\rho$ be the Perron eigenvalue of $A$, and let $\vec{r}=\langle r_1,\dotsc, r_n\rangle$ be the right eigenvector corresponding to $\rho$.   Then for each directed edge $e$ from vertex $v_i$ to vertex $v_j$, in a walk through $G$, given that we are at the vertex $v_i$, the probability of transitioning along $e$ is equal to $\frac{r_j}{r_i\rho}$.  

From this we can derive the following (which can be derived from results in \cite[Section 13.3]{LinMar21}):

\bigskip

\noindent\emph{Fact:} The probability of traversing an $n$-cycle in $G$ given that it starts at $v_1$ is equal to $\rho^{-n}$.

\bigskip

Indeed, given an $n$-cycle that begins at $v_1$, which passes through the vertices $v_1, v_{i_1}, v_{i_2},\dotsc, v_{i_{n-1}}, v_1$, the probability of transitioning along this $n$-cycle is
\[
\left(\frac{r_{i_1}}{r_1\rho}\right)\left(\frac{r_{i_2}}{r_{i_1}\rho}\right)\cdots\left(\frac{r_0}{r_{i_{n-1}}\rho}\right)=\frac{1}{\rho^n}.
\]

Using this, we can establish the following, which relates the material from the present section to the material from Section \ref{sec-coding}:

\begin{thm}
Let $T$ be a finite tree with associated pointed directed graph $\mathcal{G}=(G, v_1)$ and length function $\ell$.  Then the restriction of the Parry measure on the sofic shift $X_G$ with presentation $G$ to the path set determined by $\mathcal{G}$ is equal to the pushforward of $\mu_\ell$ under the function $\Psi_T:\Sigma_k^\infty\rightarrow F_T$, which maps coding sequences in $\Sigma_k^\infty$ to elements in $F_T$.
\end{thm}

\begin{proof}
Let $\rho$ be the Perron eigenvalue of the adjacency matrix given by $G$.  By the above fact, for the shift closure of the set of one-way infinite walks through $G$, the value that the Parry measure assigns to each cylinder set given by an $n$-cycle in $G$ that begins at $v_1$, namely the value $\rho^{-n}$.   Similarly, if we restrict the Parry measure to the path set given by $\mathcal{G}$, which by Theorem \ref{thm-treegraph} is precisely the fractal $F_T$, we get a measure $\xi$ on $F_T$ that assigns to each cylinder set defined by a $n$-cycle in $G$ that begins at $v_1$ the value $\rho^{-n}$.  Moreover, as we saw in the proof of Theorem \ref{thm-treegraph}, each $n$-cycle in $G$ that begins at $v_1$ has the same edge labels as some terminal node in $T$ of length $n$.  Thus, $\xi$ assigns to each cylinder set defined by a terminal node of $T$ the value $\rho^{-n}$.

Next, by Theorem \ref{thm-Perron-cc},  $\log(\rho)$ is the channel capacity of $\ell$. As shown at the end of Subsection \ref{subsec-ucc}, $\mu_\ell(i)=r^{\ell(i)}$, where $r=1/\rho$.  Then the pushforward of $\mu_\ell$ under the function $\Psi_T$ is a measure $\nu$ on $F_T$ that is completely determined by its behavior on cylinders given by the terminal nodes $\tau_0,\dotsc, \tau_{k-1}$ in $T$, where $\nu(\tau_i)=\mu_\ell(i)$, so that $\nu(\tau_i)=\rho^{-\ell(i)}$.  In particular, for a terminal $\tau_i$ node of length $n$, $\ell(i)=|\tau_i|=n$, so that $\nu(\tau_i)=\rho^{-n}$.  By what we derived in the previous paragraph, $\nu$ and $\xi$ assign the same values to the same cylinders, thereby establishing the theorem.
\end{proof}

%
%
%
%
%
\section{Acknowledgements} I am thankful for the feedback and suggestions made by two anonymous referees that have improved the presentation of this article.

\bibliographystyle{alpha}
\bibliography{ucc_fractal}

\end{document}